\documentclass[]{interact}

\usepackage{epstopdf}% To incorporate .eps illustrations using PDFLaTeX, etc.
\usepackage[caption=false]{subfig}% Support for small, `sub' figures and tables

\usepackage[numbers,sort&compress]{natbib}% Citation support using natbib.sty
\bibpunct[, ]{[}{]}{,}{n}{,}{,}% Citation support using natbib.sty
\makeatletter% @ becomes a letter
\def\NAT@def@citea{\def\@citea{\NAT@separator}}% Suppress spaces between citations using natbib.sty
\makeatother% @ becomes a symbol again

\theoremstyle{plain}% Theorem-like structures provided by amsthm.sty
\newtheorem{theorem}{Theorem}[section]
\newtheorem{lemma}[theorem]{Lemma}
\newtheorem{corollary}[theorem]{Corollary}

\theoremstyle{definition}

\newtheorem{example}[theorem]{Example}

\theoremstyle{remark}
\newtheorem{remark}{Remark}

\begin{document}

%\articletype{ARTICLE TEMPLATE}% Specify the article type or omit as appropriate

\title{Singularities of a characteristic Cauchy problem for a PDE with singular coefficients}

\author{
\name{Mohamed Amine Kerker\textsuperscript{a}\thanks{Email: mohamed-amine.kerker@univ-annaba.dz}}
\affil{\textsuperscript{a}Laboratory of Applied Mathematics, Badji Mokhtar University-Annaba, P.O. Box 12, Annaba 23000 Algeria}
}

\maketitle

\begin{abstract}
In this paper we give an explicit representation of the solutions of a characteristic Cauchy problem for a class of PDEs with singular coefficients. We give the explicit solutions in terms of the Gauss hypergeometric functions, which enable us to study the singularities and the analytic continuation. Our results are illustrated through some examples.
\end{abstract}

\begin{keywords}
Gauss hypergeometric function; Analytic continuation; Singular solution
\end{keywords}

\section{Introduction}
In \cite{Treves}, Treves studied the Cauchy problem for the partial differential equation
\begin{equation}
x^2u_{tt}-u_{xx}+\lambda u_t=0.\label{Trv}
\end{equation}

By using the concatenation method, he showed that the uniqueness fails if $\lambda$ is an odd positive integer. Beals and Kannai \cite{Beals} constructed exact global fundamental solutions for  a singular hyperbolic equation generalizing (\ref{Trv}): 
\begin{equation}
x^{2k-2}u_{tt}-u_{xx}+\lambda (k-1)x^{k-2} u_t=0.
\end{equation}

In \cite{Bentrad06}, Bentrad constructed singular solutions for the following equation with analytic initial data,
\begin{equation}
x^ku_{tt}-t^qx^pu_{xx}+c_1t^qx^{p-1}u_x+c_2t^qx^{p-2}u=0,\label{ben}
\end{equation}
as series with hypergeometric terms.

\vspace{0.3cm}

In this paper, we discuss the singularities of the solutions of a characteristic Cauchy problem for a class of partial differential equations with singular coefficients, which generalizes (\ref{ben}). More precisely, we will consider, in a neighborhood of the origin of $\mathbb{C}^2$, the following analytic Cauchy problem:
\begin{equation}
\left\{ 
\begin{array}{l}
L_\gamma u:=x^mL_1(t,\partial_t) u-t^nx^{p-2}L_2(x,\partial_x) u=0,  \\
u(0,x)=u_0 (x),\\
u_t(0,x)=0,\label{pb}
\end{array}
\right.
\end{equation}
where 
$$L_1(t,\partial_t)=\partial^2_t +\frac{\gamma}{t}\partial_t,$$
$$L_2(x,\partial_x)=x^2\partial^2_x +Ax\partial_x +B,$$ 
and where $m, n, p\in\mathbb{N}$, such that $q=m-p+2>0$, and $\gamma, A,B\in\mathbb{C}$.\\

We shall show that near the origin the solution of (\ref{pb}) is ramified around the union of characteristic curves:
$$
K_1 : x=0\quad \text{and }\quad K_2 : x^{q}-\left( \frac{q}{n+2}\right)^2 t^{n+2}=0.
$$

Generally it is difficult to investigate the properties of the singularities of solutions for PDEs. A natural approach is to represent the solutions explicitly, which makes the study of their singularities easier. Our method is to construct solutions in terms of Gauss hypergeometric functions (GHF for short). Since it has intrinsic singularities, the GHF was used successfully, in many papers, to construct explicit solutions and then study their singularities and analytic continuation, see e.g.\cite{Bentrad11,Bentrad06,Kerker,Tsutsui,Urabe},  and references therein.

%\medskip 

\section{Hypergeometric solutions}
We first recall some properties of the Gauss hypergeometric function, which will be used throughout this paper. Next, we reduce the equation $L_{\gamma }u=0$ to a
special ordinary differential equation, and then select those with analytic Cauchy data.
\subsection{The Gauss hypergeometric function}
The Gauss hypergeometric function plays an important role in mathematical analysis and its application. It is defined for $c\notin -\mathbb{N}$ by analytic continuation of the sum of the hypergeometric series
\begin{equation*}
F\left( a,b,c,z\right) :=\sum_{i=0}^{\infty }\frac{\left( a\right)
_{i}\left( b\right) _{i}}{\left( c\right) _{i}i!}z^{i},
\end{equation*}%
where $(\lambda)_i$ denotes the Pochhammer symbol or the shifted factorial, defined as
\begin{equation*}
\left( \lambda\right) _{i}=\frac{\Gamma \left( \lambda+i\right) }{\Gamma \left( \lambda\right) 
}=\lambda\left( \lambda+1\right) ...\left( \lambda+i-1\right) .
\end{equation*}%
It arises naturally in the solution of the Gauss hypergeometric linear differential equations, with parameters $(a,b,c)$,
$$
z(1-z)y''+\left[ c-(1+a+b)z\right] y'-aby=0,
$$
which is a Fuchsian equation with three singularities: $0,1$ and $\infty$. Furthermore, the GHF is ramified around these three singularities. Its principal branch is the one defined on the cut plane $\vert \arg(1-z)\vert<\pi$. 
\subsection{Reduction to a hypergeometric differential equation}
\begin{lemma}
\label{lemma}The equation $L_{\gamma }u=0$, with $u=x^{l}w\left( z\right) $
and $$z(t,x)=\left(\frac{q}{n+2}\right)^2\frac{t^{n+2}}{x^{q}},$$ is reduced to the Gauss hypergeometric equation with parameters $(a,b,c)$, where 
\begin{equation}
a=-\frac{\alpha+l}{q},\ b=\frac{1+\alpha-A-l}{q},\ c=\frac{n+\gamma+1}{n+2},\label{prm} 
\end{equation}
and $\alpha$ is a parameter such that $$\alpha (\alpha -A+1)=-B.$$
\end{lemma}

\begin{proof}
Let $u\left( t,x\right) =x^{l}w\left( z\right) $ with $z=(\frac{q}{n+2})^2\frac{t^{n+2}}{x^{q}}$.
Substituting $x^{l}w$ for $u$, $L_{\gamma }u=0$ we obtain: 
\begin{equation}
z\left( 1-z\right) w''+\left[ \frac{\gamma+n+1}{n+2} -\left( m-p+3-2l-A\right) \frac{z}{q}\right] w'-\frac{l\left( -1+A+l\right) }{4}w=0.  \label{eq}
\end{equation}%

Therefore, if $c \notin \mathbb{Z}$ a fundamental system of solutions of (\ref{eq}),\ for $|z|<1,$ is given by%
\begin{equation*}
w_{1}(z)=F(a,b,c ,z),
\end{equation*}
\begin{equation*}
w_{2}\left( z\right) =z^{1-c }F(1-c+a,1-c+b,2-c ,z).
\end{equation*}%
\end{proof}
\subsection{Solutions with special Cauchy data}
Let $S:t=0$ be the initial curve, and $K=K_1 \cup K_2$ with
$$
K_1 : x=0,\quad \quad K_2 : x^{q}-\left( \frac{q}{n+2}\right)^2 t^{n+2}=0,
$$
and consider in the neighborhood of the origin of $\mathbb{C}^{2}$,
\begin{equation*}
\Omega_{r}=\left\lbrace (t,x) \in\mathbb{C}^2; \left\vert x^{q}-\left( \frac{q}{n+2}\right)^2 t^{n+2}\right\vert<r\right\rbrace , 
\end{equation*}%
the following Cauchy problem 
\begin{equation*}
(\mathcal{P})\quad\left\{ 
\begin{array}{l}
P_\gamma U_l=0,\\
U_l(0,x)=x^l,\\
\partial_t U_l(0,x)=0.
\end{array}
\right.
\end{equation*}
%\begin{leftbar}
\begin{theorem}  Suppose that $\gamma$ is not a negative integer. If $c$, $c-a-b$ and $a-b\notin\mathbb{Z}$, the Cauchy problem $(\mathcal{P})$ has a unique holomorphic solution on the universal covering space $\mathcal{R}(\Omega _{r}-K)$. Moreover,  the solution has the form 
\begin{equation}
U_l(t,x)=x^lF(a,b,c,z),\label{sol}
\end{equation}%
where
$$
z=\left(\frac{q}{n+2}\right)^2\frac{t^{n+2}}{x^{q}},
$$
and the parameters $a,b$ and $c$ are given in (\ref{prm}).
\end{theorem}
%\end{leftbar}
\begin{proof}
By multiplying the equation $L_\gamma u=0$ by $tx^{-m}$, we obtain the following equivalent Cauchy problem of Fuchs type in the sense of Baouendi-Goulaouic:
$$
(\mathcal{P'})\quad\left\{ 
\begin{array}{l}
t\partial^2_tu+\gamma\partial_tu- t^{n+1}x^{-q}L_xu=0,\\
U_l(0,x)=x^l.\\

\end{array}
\right.
$$

Since $\gamma\notin\mathbb{Z}^-=\{-1,-2,...\}$, by the Baouendi-Goulaouic theorem \cite{Baouendi}, there
is a unique holomorphic solution $U_l$ to the Cauchy problem $(\mathcal{P'})$. Next, by Lemma 2.1, $U_l=x^l(c_1w_1 +c_2w_2)$, where $c_1$ and $c_2$ are arbitrary constants, solves $L_\gamma u=0$. Taking into account the Cauchy data, we obtain $c_1=1$, and $c_2=0$.

Furthermore, by construction the solution $U_{l}$ is composed of a hypergeometric function, which is holomorphic on the universal covering of $\mathbb{D}- ({0,1,\infty})$ where $ \mathbb{D}$ is the
Riemann sphere. So, the study of the ramification and the singularities of the solution is reduced to those corresponding well-known properties of GHFs. The mapping 
\begin{equation*}
z(t,x)=\left(\frac{q}{n+2}\right)^2\frac{t^{n+2}}{x^{q}}
\end{equation*}
transforms 
\begin{eqnarray*}
S: t=0 &\quad\text{into}\quad z=0, \\
K_{2}: x^{q}-\left( \frac{q}{n+2}\right)^2 t^{n+2}=0 &\quad\text{into} \quad z=1,\\
K_{1} : x=0 &\quad\text{into}\quad z=\infty. \\
\end{eqnarray*}

Further, we notice that $U_l$ does not ramify on $t=0$, $x\neq0$, because of the Cauchy-Kowalevsky theorem.
It follows that $U_{l}$ is holomorphic on the universal
covering space $\mathcal{R}(\Omega _{r}-K)$. Equivalently, $U_l$ can be extended analytically along any curve starting in $\Omega_r$ without crossing the characteristic curves $K_1$ and $K_2$. In the next subsection we give the explicit representation of the analytic continuation of $U_l$. 
\end{proof}
\begin{corollary}
Depending on various parameters, the solution of the Cauchy
problem may be holomorphic across some parts of $K$: 
\begin{enumerate}
\item When $a\in-\mathbb{N}$, we have the following results:
\begin{enumerate}
\item $U_l$ is holomorphic on $K_1$ if and only if $l+aq\in\mathbb{N}$.
\item $U_l$ is always holomorphic on $K_2$.
\end{enumerate}
\item When $b\in-\mathbb{N}$, we have the following results:
\begin{enumerate}
\item $U_l$ is holomorphic on $K_1$ if and only if $l+bq\in\mathbb{N}$.
\item $U_l$ is always holomorphic on $K_2$.
\end{enumerate}
\item When $c-a\in-\mathbb{N}$, we have the following results:
\begin{enumerate}
\item $U_l$ is holomorphic on $K_1$ if and only if $l+bq\in\mathbb{N}$.
\item $U_l$ is always ramified around $K_2$.
\end{enumerate}
\item When $c-b\in-\mathbb{N}$, we have the following results:
\begin{enumerate}
\item $U_l$ is holomorphic on $K_1$ if and only if  $l+aq\in\mathbb{N}$. 
\item $U_l$ is always ramified around $K_2$.
\end{enumerate}
\item When $a,b,c-a,c-b\notin-\mathbb{N}$, we have the following results: 
\begin{enumerate}
\item $U_l$ is holomorphic on $K_1$ if and only if $-\alpha$ and $1+\alpha-A\in\mathbb{N}$.
\item $U_l$ is always ramified around $K_2$.

\end{enumerate}
\end{enumerate}

\end{corollary}

\begin{proof} We have 
$$
U_l(t,x)=x^lF(a,b,c,z),\quad\quad z=\left(\frac{q}{n+2}\right)^2\frac{t^{n+2}}{x^{q}}.
$$
Hence, by observing the singularities of the GHF, we get
\begin{enumerate}
\item When $a=-N\in-\mathbb{N}$, $F$ reduces to polynomial of degree $N$. Precisely, we have
$$
F(-N,b,c,z)=\sum_{i=0}^N \frac{(-N)_i(b)_i}{(c)_ii!}z^i.
$$
The term of degree $N$ of $x^l F$ is
$$
a_Nx^lz^N=Ct^{N(n+2)}x^{l-N(m-p+2)}=Ct^{N(n+2)}x^{l-Nq}.
$$
Therefore, $U_l$ is 
\begin{enumerate}
\item holomorphic on $K_1$ if and only if $l+aq\in\mathbb{N}$.
\item $U_l$ is always holomorphic on $K_2$.
\end{enumerate}
\item The case when $b\in-\mathbb{N}$ is treated similarly.
\item When $c-a=-N\in-\mathbb{N}$, by the Pfaff's identity, we have
\begin{eqnarray*}
F(a,b,c,z)&=&(1-z)^{-b}F\left(c-a,b,c,\frac{z}{z-1}\right)\\
&=&(1-z)^{-b}F\left(-N,b,c,\frac{z}{z-1}\right)\\
&=&\sum_{i=0}^Na_iz^i(1-z)^{-b-i}
\end{eqnarray*}
Hence, the last term of $x^lF$ is
\begin{eqnarray*}
a_Nx^lz^N(1-z)^{-b-N}&=&Ct^{N(n+2)}x^{l-qN}\left[ x^q-\left(\frac{q}{n+2}\right)^2t^{n+2}\right]^{-b-N}x^{q(N+b)}\\
&=&Ct^{N(n+2)}x^{l+qb}\left[ x^q-\left(\frac{q}{n+2}\right)^2t^{n+2}\right]^{c-a-b}\\
\end{eqnarray*}
Therefore, $U_l$ is 
\begin{enumerate}
\item holomorphic on $K_1$ if and only if $l+bq\in\mathbb{N}$.
\item $U_l$ is always ramified around $K_2$ since $c-a-b\notin\mathbb{Z}$.
\end{enumerate}
\item The case when $c-b\in-\mathbb{N}$ is treated similarly.
\item When $a,b,c-a,c-b\notin-\mathbb{N}$, we have
\begin{enumerate}
\item By the connexion formula $(\mathcal{F}_2)$ (see the next subsection), around $z=\infty$, namely, around $K_1$, we have
\begin{eqnarray*}
U_{l}(t,x) &=& A_3x^l(1-z)^{-a}F\left(\frac{1}{1-z}\right)+A_4x^l(1-z)^{-b}F\left(\frac{1}{1-z}\right) \\
\\
&=& A_3x^{aq+l}\left[x^{q}-\left( \frac{q}{n+2}\right)^2 t^{n+2}\right]^{-a}F\left(\frac{1}{1-z}\right)\\
 &+& A_4x^{bq+l}\left[x^{q}-\left( \frac{q}{n+2}\right)^2 t^{n+2}\right]^{-b}F\left(\frac{1}{1-z}\right)\\
&=& A_3x^{-\alpha}\left[x^{q}-\left( \frac{q}{n+2}\right)^2 t^{n+2}\right]^{-a}F\left(\frac{1}{1-z}\right)\\
 &+&A_4x^{1+\alpha-A}\left[x^{q}-\left( \frac{q}{n+2}\right)^2 t^{n+2}\right]^{-b}F\left(\frac{1}{1-z}\right).
\end{eqnarray*}
%\begin{eqnarray*}
%a 
%\end{eqnarray*}%
Therefore, $U_l$ is holomorphic on $K_1$ if and only if $-\alpha$ and $1+\alpha-A\in\mathbb{N}$.
\item By the connexion formula $(\mathcal{F}_1)$, around $z=1$, namely, around $K_2$, we observe that, since $c-a-b\notin\mathbb{Z}$, $(1-z)^{c-a-b}$ is always ramified around $z=1$, and then $U_l$ is ramified around $K_2$. 
\end{enumerate}
\end{enumerate}
\end{proof}
Here are some illustrative examples:
\begin{example}
Consider in $\mathbb{C}^{2}$, the Cauchy problem 
\begin{equation*}
\left\{ 
\begin{array}{l}
x^4\left(\partial^2_t u+\frac{1}{3t}\partial_t u\right)-tx\left( x^2\partial^2_x u-\frac{x}{2}\partial_x u\right)  =0,  \\
u(0,x)=x^3,\\
u_t(0,x)=0.
\end{array}
\right.
\end{equation*}
The solution $u(t,x)$ is given by: 
\begin{equation*}
U_3(t,x)=x^3+\frac{2}{7}t^3 . 
\end{equation*}
We observe that $U_3$ is holomorphic.
\end{example}

\begin{example}
Consider in $\mathbb{C}^{2}$, the Cauchy problem 
\begin{equation*}
\left\{ 
\begin{array}{l}
x^3\left(\partial^2_t u+\frac{1}{2t}\partial_t u\right)-t\left( x^2\partial^2_x u+\frac{x}{2}\partial_x u-u\right)  =0,  \\
u(0,x)=x^2,\\
u_t(0,x)=0.
\end{array}
\right.
\end{equation*}
The solution $U_2(t,x)$ is given by: 
\begin{equation*}
U_2(t,x)= x^2+\frac{2t^3}{5x} . 
\end{equation*}
We observe that $U_2$ is singular on $K_1 : x=0$.
\end{example}

\begin{example}
Consider in $\mathbb{C}^{2}$, the Cauchy problem 
\begin{equation*}
\left\{ 
\begin{array}{l}
x^3\left(\partial^2_t u-\frac{1}{3t}\partial_t u\right)-x^2\partial^2_x u+x\partial_x u =0,  \\
u(0,x)=x,\\
u_t(0,x)=0.
\end{array}
\right.
\end{equation*}
The solution $U_1$ is given by: 
\begin{equation*}
U_1(t,x)= x(1-\frac{z}{2})(1-z)^{-\frac{2}{3}},\ \text{where}\ z=\frac{9t^2}{4x^3} .
\end{equation*}
$U_1$ is ramified around $K_{2}: 4x^3-9t^2=0 $.
\end{example}

\begin{example}
The unique solution of the following  Cauchy problem 
\begin{equation*}
\left\{ 
\begin{array}{l}
x^3\left(\partial^2_t u-(2t)^{-1}\partial_t u\right)-t\left( x^2\partial^2_x u+3x\partial_x u-\frac{9}{4} u\right)  =0,  \\
u(0,x)=x^2,\\
u_t(0,x)=0
\end{array}
\right.
\end{equation*} is given by: 
\begin{equation*}
U_2(t,x)= \frac{(x^3-t^3)^{\frac{5}{6}}}{\sqrt{x}}. 
\end{equation*}
$U_2$ is singular on both $K_1 : x=0$ and $K_{2}: x^3-t^3=0 $.
\end{example}

\begin{remark}
When $A=B=n=0$ and $m=p$, the equation $L_\gamma u=0$ reduces to the Euler-Poisson-Darboux equation  
$$
\partial^2_t u-\partial^2_x u+\frac{\gamma}{t} \partial_t u=0.
$$
The form (\ref{sol}) becomes
$$
U_l(t,x)=x^lF\left(-\frac{l}{2},\frac{1-l}{2},\frac{\gamma+1}{2},\frac{t^2}{x^2}\right).
$$
Furthermore, if $\gamma=0$, we get the wave equation
$$
\partial^2_t u-\partial^2_x u=0.
$$
In this case, the solution is given by
$$
U_l(t,x)=x^lF\left(-\frac{l}{2},\frac{1-l}{2},\frac{1}{2},\frac{t^2}{x^2}\right),
$$
which reduces, by applying formula 15.1.9 of \cite{Abramowitz}
$$
F(a,a+\frac{1}{2},\frac{1}{2},z^2)=\frac{1}{2}\left[(1+z)^{-2a}+(1-z)^{-2a}\right],
$$
to the form
$$
U_l(t,x)=\frac{1}{2}\left[(x+t)^l+(x-t)^l\right],
$$
which is the well-known D'Alembert's formula for the problem $(\mathcal{P})$.

\end{remark}

\begin{remark}
If $\gamma \in\mathbb{Z}^{-}$, a null solution appears so that the uniqueness of the solution of $(\mathcal{P})$ fails. The
solutions take the form $$U_l(t,x)= \bar{U_l} + t^{1-\gamma}V(t,x),$$ where $\bar{U_l}$
is a particular solution of $(\mathcal{P})$, and $V$ is a solution of $L_{2-\gamma}u=0$. 
\end{remark}
\begin{example}
For any $\lambda\in\mathbb{C}$, $(x^3-t^3)^{\frac{1}{3}}+\lambda t^2$ is a solution of the following Cauchy problem 
\begin{equation*}
\left\{ 
\begin{array}{l}
x^4\left(\partial^2_t u-\frac{1}{t}\partial_t u\right)-tx\left( x^2\partial^2_x u-x\partial_x u\right)  =0,  \\
u(0,x)=x,\\
u_t(0,x)=0.
\end{array}
\right.
\end{equation*}

\end{example}

\subsection{Analytical continuation and ramification}

Using the connection formulas between the neighborhoods of the regular
singular points of the hypergeometric equation $0,1$ and $\infty$~(see \cite[%
p. 559]{Abramowitz}), we have for $\left\vert \arg\left( 1-z\right)
\right\vert $ $<\pi$:
\begin{eqnarray*}
\text{In} \left\vert 1-z\right\vert <1:&&\\
(\mathcal{F}_1) : \quad U_{l}(t,x) & = & A_{1}x^lF\left( a,b,1+a+b-c,1-z\right)\\
& + &  A_{2}x^l\left( 1-z\right) ^{c-a-b}F\left(
c-a,c-b,c-a-b+1,1-z\right),\\
\text{In} \left\vert 1-z\right\vert >1:&&\\
(\mathcal{F}_2) : \quad U_{l}(t,x) & = & A_{3}x^l\left(1-z\right) ^{-a}F\left(
a,c-b,1+a-b,\frac{1}{1-z}\right)\\
 &+&A_{4}x^l\left( 1-z\right)
^{-b}F\left( b,c-a,1-a+b,\frac{1}{1-z}\right),
\end{eqnarray*}
where the different constants are given by%
\begin{equation*}
\begin{array}{lcl}
A_{1}=\dfrac{\Gamma \left( c \right) \Gamma \left(
c -a-b\right) }{\Gamma \left( c -a\right) \Gamma \left( c
-b\right) }, &  & A_{2}=\dfrac{\Gamma \left( c \right)\Gamma \left( a+b-c
\right) }{\Gamma \left( a\right) \Gamma \left(
b\right) }, \\ 
&  &  \\ 
A_{3}=\dfrac{\Gamma \left(
c \right)\Gamma \left( b-a\right) }{\Gamma \left( b\right)\Gamma \left( c -a\right)  }, & 
& A_{4}=\dfrac{\Gamma \left(
c \right)\Gamma \left( a-b\right) }{\Gamma \left( a\right)\Gamma \left( c -b\right)  }. \\  
\end{array}%
\end{equation*}

\medskip
\noindent Formulas $(\mathcal{F}_1)$ and $(\mathcal{F}_2)$ enable us to study the ramification of $U_l$ around the characteristic surfaces. Let $P$ be a point belonging to $\Omega_r-K$ such that $\arg \left( 1- z(P)\right) =0$, with $z(P)\in(0,1)$, and let $\lambda_1$ and $\lambda_2$ be loops with basepoint $P$, which encircle $K_1$ and $K_2$, in the positive sence, respectively. Then, we have: 
\begin{eqnarray*}
U_l(\lambda_1(P))& =&U^{(1)}_{l}\left( P\right) +e^{2\pi i\left( c-a-b\right)
}U^{(2)}_l\left( P\right) , \\
\\
U_l(\lambda_2(P))& =&e^{2\pi ia}U^{(3)}_l\left( P\right) +e^{2\pi
ib}U^{(4)}_l\left( P\right),
\end{eqnarray*}%
where $U^{(i)}_l$ are the values of initial branches. 
\section{Series solutions}
Let 
$$
u_0(x)=\sum_{l=0}^{\infty}a_lx^l
$$
be an analytic function with radius of convergence $R>0$. The Cauchy problem (\ref{pb}) has a unique solution, which is given by
\begin{equation}
u(t,x)=\sum_{l=0}^{\infty}a_lU_l(t,x).\label{S}
\end{equation}
We focus here on the convergence of (\ref{S}).

Hereafter, we agree to use the following notations to describe majorant relations. We say that the formal power series $$g(z)=\sum_{l=0}^{\infty}A_lz^l$$ majorants the formal power series $$f(z)=\sum_{l=0}^{\infty}a_lz^l$$ if $$\vert a_l\vert\leq A_l,\ \forall l\geq0,$$ and then we use the Poincar\'e's notation: $f\ll g$.

To prove the convergence of (\ref{S}) we use the following lemma:

\begin{lemma}\label{lemma} If $a\geq b>c>0$, then 
\begin{equation*}
F(a,b,c;z)\ll \frac{\Gamma (c)\Gamma (a+b-c)}{\Gamma (a)\Gamma (b)}%
(1-z)^{c-a-b}.
\end{equation*}
\end{lemma}

For the proof of this lemma see \cite{Bentrad09,Ponnusamy}.
\begin{theorem} The series (\ref{S}) converges for
$$\left\vert x^{q}-\left( \frac{q}{n+2}\right)^2 t^{n+2}\right\vert<\frac{R^{q}}{4}.$$ 
\end{theorem}
\begin{proof} 
By applying the Pfaff's identity to the hypergeometric part of $U_l$, we get
\begin{equation*}
F(a,b,c,z)=(1-z)^{-b}F\left(c-a,b,c,\frac{z}{z-1}\right). 
\end{equation*}
From the observations,
$$
\vert(a)_n\vert\leq (\vert a\vert)_n
$$
and
$$
\frac{1}{\vert(c)_n\vert}\leq \frac{1}{(\vert c+1\vert-1)_n},
$$
it follows that 
$$
F\left(c-a,b,c,\frac{z}{z-1}\right)\ll F\left(\frac{l}{q}+\eta_1,\frac{l}{q}+\eta_1,\eta_3,\frac{z}{z-1}\right),
$$
where
$$
\eta_1=\frac{\vert\alpha\vert }{q}+\vert c\vert,\ \eta_2=\frac{\vert1+\alpha-A\vert }{q},\ \eta_3=\vert c+1\vert-1.
$$
Then, by applying the Lemma \ref{lemma}, we obtain
$$
F\left(c-a,b,c,\frac{z}{z-1}\right)\ll C_l(1-z)^{\frac{2l}{q}+\eta_1+\eta_2-\eta_3},
$$
where 
\begin{equation*}
C_l=\frac{\Gamma(\eta_3)\Gamma(\frac{2l}{q}+\eta_1+\eta_2-\eta_3)}{\Gamma(\frac{l}{q}+\eta_1)\Gamma(\frac{l}{q}+\eta_2)}. 
\end{equation*}
Consequently, the hypergeometric part is estimated as follows
$$
F(a,b,c,z)\ll C_l(1-z)^{\frac{l}{q}+\eta_4},
$$
where 
$$
\eta_4=\eta_1+\eta_2-\eta_3-\frac{1+\alpha-A}{q}.
$$
Stirling's formula gives $C_l=\mathcal{O}\left( 2^{\frac{2l}{q}}\right) $ for $l$
large. Therefore, 
\begin{equation*}
\limsup_{l\rightarrow\infty}|U_{l}|^{1/l}\leq2^{\frac{2}{q}}\left\vert x^{q}-\left( \frac{q}{n+2}\right)^2 t^{n+2}\right\vert^{\frac{1}{q}}.
\end{equation*}
It follows that the series (\ref{S}) converges for 
$$2^{\frac{2}{q}}\left\vert x^{q}-\left( \frac{q}{n+2}\right)^2 t^{n+2}\right\vert^{\frac{1}{q}}<R,$$ or equivalently, for
$$\left\vert x^{q}-\left( \frac{q}{n+2}\right)^2 t^{n+2}\right\vert<\frac{R^{q}}{4}.$$ 
\end{proof}

\end{document}